\documentclass[11pt,a4paper]{article}
\usepackage{multicol,graphicx}
\usepackage{amssymb, amsmath, amsthm}
\usepackage{amsrefs}
\usepackage{dsfont}
\usepackage[arrow, matrix, curve]{xy}
\usepackage{hyperref}
\newtheorem{Lemma}{Lemma}
\newtheorem{Proposition}{Proposition}
\newtheorem{Definition}{Definition}
\newtheorem{Theorem}{Theorem}
\newtheorem{Corollary}{Corollary}
\newtheorem{Remark}{Remark}

\DeclareMathOperator{\GL}{GL}
\DeclareMathOperator{\Gal}{Gal}

\DeclareMathOperator{\Quot}{Quot}

\DeclareMathOperator{\m-Spec}{m-Spec}
\newcommand{\Qp}{\mathbb{Q}_p}
\newcommand{\Zp}{\mathbb{Z}_p}
\newcommand{\Fp}{\mathbb F_p}

\newcommand{\ZZ}{\mathbb Z}
\newcommand{\NN}{\mathbb N}
\newcommand{\OO}{\mathcal O}
\newcommand{\mm}{\mathfrak m}
\newcommand{\nn}{\mathfrak n}
\newcommand{\pp}{\mathfrak p}

\begin{document}

\title{Hilbert-Samuel multiplicities of certain deformation rings}
\author{Fabian Sander}
\maketitle

\abstract{We compute presentations of crystalline framed deformation rings of a two dimensional representation $\bar{\rho}$ of the absolute Galois group of $\Qp$, when $\bar{\rho}$ has scalar semi-simplification,
the Hodge-Tate weights are small and $p>2$. In the non-trivial cases, we show that the special fibre is geometrically irreducible, generically reduced and the Hilbert-Samuel multiplicity is either $1$, $2$ or $4$ 
depending on $\bar{\rho}$. We show that in the last two cases the deformation ring is not Cohen-Macaulay.}

\medskip

\section{Introduction}
Let $p>2$ be a prime. Let $k$ be a finite field of characteristic $p$, $E$ be a finite totally ramified extension of $W(k)[\frac{1}{p}]$ with ring of integers $\OO$ and uniformizer $\pi.$ For a given continuous representation $\bar{\rho}\colon G_{\Qp}\rightarrow \GL_2(k)$ we consider the universal framed deformation ring $R_{\bar{\rho}}^{\square}$ and the universal framed deformation $\rho^{univ}\colon G_{\Qp}\rightarrow \GL_2(R_{\bar{\rho}}^{\square}).$ For all $\pp\in \m-Spec(R_{\bar{\rho}}^{\square}[\frac{1}{p}]),$ the set of maximal ideals of $R_{\bar{\rho}}^{\square}[\frac{1}{p}],$ we can specialize the universal representation at $\pp$ to obtain the representation \[\rho_{\pp}\colon G_{\Qp}\rightarrow \GL_2(R_{\bar{\rho}}^{\square}[\frac{1}{p}]/\pp),\] where $R_{\bar{\rho}}^{\square}[\frac{1}{p}]/\pp$ is a finite extension of $\Qp.$ Let $\tau\colon I_{ G_{\Qp}}\rightarrow \GL_2(E)$ be a representation with an open kernel, where $I_{ G_{\Qp}}$ is the inertia subgroup of $G_{\Qp}.$ We also fix integers $a,b$ with $b\ge0$ and a continuous character $\psi\colon G_{\Qp}\rightarrow \OO^{\times}$ such that $\overline{\psi\epsilon}=\det(\bar{\rho}),$ where $\epsilon$ is the cyclotomic character. Kisin showed in \cite{MR2373358} that there exist unique reduced $\OO$-torsion free quotients $R_{\bar{\rho}}^{\square,\psi}(a,b,\tau)$ and $R_{\bar{\rho},cris}^{\square,\psi}(a,b,\tau)$ of $R_{\bar{\rho}}^{\square}$ with the property that $\rho_{\pp}$ factors through $R_{\bar{\rho}}^{\square,\psi}(a,b,\tau)$ resp. $R_{\bar{\rho},cris}^{\square,\psi}(a,b,\tau)$ if and only if $\rho_{\pp}$ is potentially semi-stable resp. potentially crystalline with Hodge-Tate weights $(a,a+b+1)$ and has determinant $\psi\epsilon$ and inertial type $\tau.$
If $\tau$ is trivial then $R_{\bar{\rho},cris}^{\square,\psi}(a,b):=R_{\bar{\rho},cris}^{\square,\psi}(a,b,\mathds{1}\oplus\mathds{1})$ parametrizes all the crystalline lifts of $\bar{\rho}$ with Hodge-Tate weights $(a,a+b+1)$ and determinant $\psi\epsilon.$
The Breuil-M\'ezard conjecture, proved by Kisin for almost all $\bar{\rho},$ see also \cite{MR1944572}, \cite{BM2}, \cite{MR3134019}, \cite{HU1}, \cite{Pa1}, says that the Hilbert-Samuel multiplicity of the ring $R_{\bar{\rho}}^{\square,\psi}(a,b,\tau)/\pi$ can be determined by computing certain automorphic multiplicities, which do not depend on $\bar{\rho}$, and the Hilbert-Samuel multiplicities of $R_{\bar{\rho},cris}^{\square,\psi}(a,b)$ in low weights for $0\leq a\leq p-2, 0\leq b\leq p-1.$ For most $\bar{\rho},$ the Hilbert-Samuel multiplicities of $R_{\bar{\rho},cris}^{\square,\psi}(a,b)$ have already been determined. Our goal in this paper is to compute the Hilbert-Samuel multiplicity of the ring $R_{\bar{\rho},cris}^{\square,\psi}(a,b)$ with $0\leq a\leq p-2, 0\leq b\leq p-1$ when 

$$\bar{\rho}\colon G_{\Qp}\rightarrow \GL_2(k), \quad g\mapsto \begin{pmatrix} \chi(g)&\phi(g)\\0&\chi(g)\end{pmatrix}.$$

One may show that $R_{\bar{\rho},cris}^{\square.\psi}(a,b)$  is zero if either $b\neq p-2$ or the restriction 
of $\chi$ to $I_{\Qp}$ is not equal to $\epsilon^a$ modulo $\pi$.  

\begin{Theorem}
Let $a$ be an integer with $0\leq a\leq p-2$ such that $\chi|_{I_{\Qp}}\equiv \epsilon^a \pmod{\pi}$. Then $R_{\bar{\rho},cris}^{\square,\psi}(a,p-2)/\pi$ is geometrically irreducible, generically reduced and
\begin{align*}
e(R_{\bar{\rho},cris}^{\square,\psi}(a,p-2)/\pi)=\left\{
\begin{array}{l}
1,\ \mbox{if}\ \bar{\rho}\otimes \chi^{-1}\ \mbox{is ramified}\\ 
2,\ \mbox{if}\ \bar{\rho}\otimes \chi^{-1}\ \mbox{is unramified, indecomposable}\\
4,\ \mbox{if}\ \bar{\rho}\otimes \chi^{-1}\ \mbox{is split}
\end{array}
\right.
\end{align*}
In the last two cases, $R_{\bar{\rho},cris}^{\square,\psi}(a,p-2)$ is not Cohen-Macaulay.
\end{Theorem}
The multiplicity $4$ does not seem to have been anticipated in the literature, see for example \cite[1.1.6]{MR2505297}.
Our method is elementary in the sense that we do not use any integral $p$-adic Hodge theory. The only $p$-adic Hodge theoretic input is that if $\rho$ is a crystalline lift of $\bar{\rho}$ with Hodge-Tate weights $(0,p-1),$ then we have an exact sequence 
\[
\begin{xy}
\xymatrix{
0\ar[r]&\epsilon^{p-1}\chi_1\ar[r]&\rho\ar[r]&\chi_2\ar[r]&0,
}
\end{xy}
\]
where $\chi_1,\chi_2\colon G_{\Qp}\rightarrow \OO^{\times}$ are unramified characters. This allows us to convert the problem into a linear algebra problem, which we solve in Lemma \ref{Lemma 2}. This gives us an explicit presentation of the ring $R_{\bar{\rho},cris}^{\square,\psi}(a,p-2)$, using which we compute the multiplicities in \S\ref{sec:Mult}. Our argument gives a proof of the existence of $R_{\bar{\rho},cris}^{\square.\psi}(a,p-2)$
independent of \cite{MR2373358}.
After writing this note we discovered that the idea to convert the problem into linear algebra already appears in \cite{Sn1}.

\section{The universal ring}
After twisting we may assume that $\chi=1$ and $a=0$ so that \[\bar{\rho}(g)= \begin{pmatrix} 1 & \phi(g) \\ 0 & 1 \end{pmatrix}.\]
Since the image of $\bar{\rho}$ in $\GL_2(k)$ is a $p$-group, the universal representation factors through the maximal pro-$p$ quotient of $G_{\Qp},$ which we denote by $G.$ We have the following commuting diagram
\[
\begin{xy}
\xymatrix{
G_{\Qp}\ar[r]\ar[d]& G\ar[d]\\
G_{\Qp}^{ab}\ar[r]& G_{\Qp}^{ab}(p)\cong G^{ab}
}
\end{xy}
\]
where $G_{\Qp}^{ab}:=\Gal(\Qp^{ab}/\Qp)$ is the maximal abelian quotient of $G_{\Qp}$ and can be described by the exact sequence
\[
\begin{xy}
\xymatrix{
1\ar[r]&\Gal(\Qp^{ab}/\Qp^{ur})\ar[r]&G_{\Qp}^{ab}\ar[r]&G_{\Fp}\ar[r]&1
}
\end{xy}
\]
where $\Qp^{ur}$ is the maximal unramified extension of $\Qp$ inside $\bar{\mathbb{Q}}_p.$
Local class field theory implies that the natural map \[G_{\Qp}^{ab}\rightarrow \Gal(\Qp^{ur}/\Qp)\times\Gal(\Qp(\mu_{p^{\infty}})/\Qp)\] is an isomorphism, where $\mu_{p^{\infty}}$ is the group of $p$-power order roots of unity in $\bar{\mathbb{Q}}_p.$ The cyclotomic character $\epsilon$ induces an isomorphism \[\Gal(\Qp(\mu_{p^{\infty}})/\Qp)\xrightarrow[\epsilon]{\cong}\Zp^{\times}\] and $\Gal(\Qp^{ur}/\Qp)\cong \hat \ZZ,$ hence \[G^{ab}\cong (1+p\Zp)\times \Zp,\] where the map onto the first factor is given by $\epsilon^{p-1}.$ We choose a pair of generators $\bar{\gamma},\bar{\delta}$ of $G^{ab}$ such that $\bar{\gamma}\mapsto (1+p,0)$ and $\bar{\delta}\mapsto (1,1).$
With \cite[Lemma 3.2]{MR1757878} we obtain that $G$ is a free pro-$p$ group in two letters $\gamma,\delta$ which project to $\bar{\gamma},\bar{\delta}.$
The way we choose these generators will be of importance in the following.
\begin{Lemma}\label{Lemma 1}
Let $\eta\colon G_{\Qp}\rightarrow \Zp^{\times}$ be a continuous character such that $\eta\equiv 1 (p).$ Then $\eta=\epsilon^k\chi$ for an unramified character $\chi$ if and only if $\eta(\gamma)=\epsilon(\gamma)^k$ and $p-1|k.$
\end{Lemma}
\begin{proof}
$"\Rightarrow":$ Since $\gamma$ maps to identity in $\Gal(\Qp^{ur}/\Qp),$ we clearly have $\chi(\gamma)=1$ for every unramified character $\chi.$ Hence $\epsilon(\gamma)^k\equiv 1 (p),$ which implies $p-1|k.$

$"\Leftarrow":$ From $\eta\epsilon^{-k}(\gamma)=1$ and the fact that $\delta$ maps to the image of identity in the maximal pro-$p$ quotient of $\Gal(\Qp(\mu_{p^{\infty}})/\Qp),$ we see that $\eta\epsilon^{-k}=\chi$ for an unramified character $\chi.$
\end{proof}
Since $G$ is a free pro-$p$ group generated by $\gamma$ and $\delta,$ to give a framed deformation of $\bar{\rho}$ to $(A,\mm_A)$ is equivalent to give two matrices in $\GL_2(A)$ which reduce to $\bar{\rho}(\gamma)$ and $\bar{\rho}(\delta)$ modulo $\mm_A.$ Thus
\[R_{\bar{\rho}}^{\square}=\OO[[x_{11},\hat x_{12},x_{21},t_{\gamma},y_{11},\hat y_{12},y_{21},t_{\delta}]]\] and the universal framed deformation is given by
\begin{gather*}
\rho^{univ}\colon G\rightarrow \GL_2(R_{\bar{\rho}}^{\square})\\
\gamma\mapsto \begin{pmatrix} 1+t_{\gamma}+x_{11}&x_{12}\\x_{21}&1+t_{\gamma}-x_{11}\end{pmatrix}\\
\delta\mapsto \begin{pmatrix} 1+t_{\delta}+y_{11}&y_{12}\\y_{21}&1+t_{\delta}-y_{11}\end{pmatrix}
\end{gather*}
where $x_{12}:=\hat x_{12}+[\phi(\gamma)],\ y_{12}:=\hat y_{12}+[\phi(\delta)]$ where $[\phi(\gamma)],[\phi(\delta)]$ denote the Teichm\"uller lifts of $\phi(\gamma)$ and $\phi(\delta)$ to $\OO.$
\begin{Remark}\label{Remark 1}
We note that there are essentially $3$ different cases:
\begin{enumerate}
\item $\bar{\rho}$ is ramified $\Leftrightarrow \phi(\gamma)\neq 0 \Leftrightarrow x_{12}\in (R_{\bar{\rho}}^{\square})^{\times};$
\item $\bar{\rho}$ is unramified, non-split $\Leftrightarrow \phi(\gamma)=0, \phi(\delta)\neq 0 \Leftrightarrow x_{12}\in \mm_{R_{\bar{\rho}}^{\square}},y_{12}\in (R_{\bar{\rho}}^{\square})^{\times};$
\item $\bar{\rho}$ is split $\Leftrightarrow \phi(\gamma)=0, \phi(\delta)=0 \Leftrightarrow x_{12},y_{12}\in \mm_{R_{\bar{\rho}}^{\square}}.$
\end{enumerate}
\end{Remark}

Let $\psi\colon G_{\Qp}\rightarrow \OO^{\times}$ be a continuous character, such that $\det(\bar{\rho})=\overline{\psi\epsilon},$ and let $R_{\bar{\rho}}^{\square, \psi}$ be the quotient of $R_{\bar{\rho}}^{\square}$ which parametrizes lifts of $\bar{\rho}$ with determinant $\psi\epsilon.$ Since $\gamma,\delta$ generate $G$ as a group, we obtain
\begin{align*}
R_{\bar{\rho}}^{\square,\psi}&\cong R_{\bar{\rho}}^{\square}/(\det(\rho^{univ}(\gamma)-\psi\epsilon(\gamma)),\det(\rho^{univ}(\delta)-\psi\epsilon(\delta)))\\
&\cong \OO[[x_{11},\hat x_{12},x_{21},y_{11},\hat y_{12},y_{21}]],
\end{align*}
because we can eliminate the parameters $t_{\gamma},t_{\delta}$ due to the relations $(1+t_{\gamma})^2=\psi\epsilon(\gamma)+x_{11}^2+x_{12}x_{21},t_{\gamma}\equiv 0 (\mm), (1+t_{\delta})^2=\psi\epsilon(\delta)+y_{11}^2+y_{12}y_{21},t_{\delta}\equiv 0 (\mm).$ We let $v:=\frac{1-\epsilon^{p-1}(\gamma)}{2}$ and define four polynomials
\begin{align}
I_1:=&\ (v+x_{11})(v-x_{11})-x_{12}x_{21}\label{eq:1}\\
I_2:=&\ (v+x_{11})^2y_{12}-2(v+x_{11})x_{12}y_{11}-x_{12}^2y_{21}\label{eq:2}\\ 
I_3:=&\ x_{21}^2y_{12}-2x_{12}(v-x_{11})y_{11}-(v-x_{11})^2y_{21}\label{eq:3}\\ 
I_4:=&\ (v+x_{11})x_{21}y_{12}-2x_{12}x_{21}y_{11}-x_{12}(v-x_{11})y_{21}.\label{eq:4}
\end{align}
Since for every representation with Hodge-Tate weights $(0,p-1)$ the determinant is a character of Hodge-Tate weight $p-1$ and $R_{\bar{\rho},cris}^{\square,\psi}(0,p-2)$ parametrizes all lifts $\rho_{\pp}$ with determinant $\psi\epsilon,$ we let from now on $\psi$ have Hodge-Tate weight $p-2,$ as otherwise $R_{\bar{\rho},cris}^{\square,\psi}(0,p-2)$ would be trivial. 
\begin{Definition}
We set \[R:=R_{\bar{\rho}}^{\square,\psi}/(I_1,I_2,I_3,I_4).\]
\end{Definition}
Our goal is to show that $R_{\bar{\rho},cris}^{\square,\psi}(0,p-2)$ is isomorphic to $R.$
\begin{Lemma}\label{Lemma 2}
If $\pp\in\m-Spec(R_{\bar{\rho}}^{\square,\psi}[\frac{1}{p}]),$ then $\pp\in \m-Spec(R[\frac{1}{p}])$ if and only if $\rho_{\pp}$ is reducible and $\rho_{\pp}(\gamma)$ acts on the $G$-invariant subspace with eigenvalue $\epsilon^{p-1}(\gamma).$ 
\end{Lemma} 
\begin{proof}
Let $\pp\in\m-Spec(R_{\bar{\rho}}^{\square,\psi}[\frac{1}{p}]),$ such that $\rho_{\pp}$ is reducible and $\rho_{\pp}(\gamma)$ acts on the $G$-invariant subspace with eigenvalue $\epsilon^{p-1}(\gamma).$ Since $\det(\rho_{\pp}(\gamma))=\psi\epsilon(\gamma)=\epsilon(\gamma)^{p-1}$ and $\epsilon(\gamma)^{p-1}$ is an eigenvalue of $\rho_{\pp}(\gamma),$ the other eigenvalue must be $1.$ Therefore we can write $1+t_{\gamma}=\frac{\epsilon(\gamma)^{p-1}+1}{2}$ and obtain 
\begin{align*}
0&=\det\begin{pmatrix} 1+t_{\gamma}+x_{11}-\epsilon(\gamma)^{p-1}&x_{12}\\x_{21}&1+t_{\gamma}-x_{11}-\epsilon(\gamma)^{p-1}\end{pmatrix}\\
&=(v+x_{11})(v-x_{11})-x_{12}x_{21}.
\end{align*}
If we now take $\pp$ as above but with $I_1:=(v+x_{11})(v-x_{11})-x_{12}x_{21}\subset \pp,$ it is easy to see that the vectors $v_1=\begin{pmatrix} -x_{12} \\ v+x_{11} \end{pmatrix}$ and $v_2=\begin{pmatrix} v-x_{11} \\ -x_{21} \end{pmatrix}$ are eigenvectors for $\rho_{\pp}(\gamma)$ with eigenvalue $\epsilon(\gamma)^{p-1}$ if they are non-zero. But at least one of them is non-zero because otherwise we obtain $v=0$ and thus $\epsilon(\gamma)^{p-1}=1,$ which is a contradiction to the definition of $\gamma.$ So $\rho_{\pp}$ is reducible with an invariant subspace on which $\rho_{\pp}(\gamma)$ acts by $\epsilon(\gamma)^{p-1}$ if and only if the vectors $v_1,v_2,\rho^{univ}(\delta)v_1,\rho^{univ}(\delta)v_2$ are pairwise linear dependent. It is easy to check that this is equivalent to the satisfaction of the equations $I_1=I_2=I_3=I_4=0.$
\end{proof}

\begin{Lemma}\label{Lemma 3}
\[\m-Spec(R[\dfrac{1}{p}])=\m-Spec(R_{\bar{\rho}}^{\square,\psi}(0,p-2)[\dfrac{1}{p}])\]
\end{Lemma}
\begin{proof}
From \cite[Prop.3.5(i)]{MR2551764} we know that every crystalline lift $\rho_{\pp}$ of a reducible 2-dimensional representation $\bar{\rho},$ such that $\rho_{\pp}$ has Hodge-Tate-weights $(0,p-1),$ is reducible itself. Moreover, \cite[Thm. 8.3.5]{BC} says that if $\rho$ is a reducible 2-dimensional crystalline representation, then we have an exact sequence
\[
\begin{xy}
\xymatrix{
0\ar[r]&\epsilon^{p-1}\chi_1\ar[r]&\rho\ar[r]&\chi_2\ar[r]&0.
}
\end{xy}
\]
Thus $\rho_{\pp}(\gamma)$ acts on the invariant subspace as $\epsilon(\gamma)^{p-1}$ and hence from Lemma \ref{Lemma 2} it is clear that \[\m-Spec(R[\dfrac{1}{p}])\supset\m-Spec(R_{\bar{\rho}}^{\square,\psi}(0,p-2)[\dfrac{1}{p}]).\] For the other inclusion we note that it is also clear from Lemma \ref{Lemma 2} that any maximal ideal $\pp\in \m-Spec(R[\frac{1}{p}])$ gives rise to a reducible representation $\rho_{\pp}$ such that $\rho_{\pp}(\gamma)$ acts on the invariant subspace as $\epsilon(\gamma)^{p-1}$ and that the other eigenvalue of $\rho_{\pp}(\gamma)$ is $1.$ So we obtain with Lemma \ref{Lemma 1} that 
$\rho_{\pp}$ is an extension of two crystalline characters \[0\rightarrow \eta_1 \rightarrow * \rightarrow \eta_2 \rightarrow 0\] where the Hodge-Tate-weight of $\eta_1$ is equal to $p-1$ and the weight of $\eta_2$ is equal to $0.$ Then we can conclude from \cite[Prop. 1.28]{MR1263527} that it is semi-stable and from \cite[Thm. 8.3.5, Prop. 8.3.8]{BC} that it is crystalline and hence $\pp\in \m-Spec(R_{\bar{\rho}}^{\square,\psi}(0,p-2)[\frac{1}{p}]).$
\end{proof}
\begin{Remark}
We have the following identities mod $I_1$:
\begin{align}
x_{21}I_2=&(v+x_{11})I_4\label{eq:5}\\ 
(v-x_{11})I_2=&x_{12}I_4\label{eq:6}\\ 
x_{21}I_4=&(v+x_{11})I_3\label{eq:7}\\ 
(v-x_{11})I_4=&x_{12}I_3.\label{eq:8}
\end{align}
\end{Remark}

\section{Reducedness}
In order to show that $R_{\bar{\rho}}^{\square,\psi}(0,p-2)$ is equal to $R,$ it is enough to show that $R$ is reduced and $\OO$-torsion free, since then the assertion follows from Lemma \ref{Lemma 3}, as $R[\frac{1}{p}]$ is Jacobson because $R$ is a quotient of a formal power series ring over a complete discrete valuation ring.
\begin{Lemma}\label{Lemma 4}
If $\OO=W(k)$, then $R$ is an $W(k)$-torsion-free integral domain.
\end{Lemma}
\begin{proof}
We distinguish two cases.

If $\bar{\rho}$ is ramified, i.e. $x_{12}$ is invertible, we consider the fact that for every complete local ring $A$ with $a\in \mm_A,u\in A^{\times},$ there is a canonical isomorphism $A[[z]]/(uz-a)\cong A.$ Using this we see from \eqref{eq:1},\eqref{eq:2},\eqref{eq:6} and \eqref{eq:8} that
\begin{align*}
R=&\OO[[x_{11},\hat x_{12},x_{21},y_{11},\hat y_{12},y_{21}]]/(I_1,I_2)\\
\cong&\OO[[x_{11},\hat x_{12},y_{11},\hat y_{12}]],
\end{align*}
which shows the claim.

In the second case, where $\bar{\rho}$ is unramified, i.e. $x_{12}\notin R^{\times},$ we consider the ideal $I:=(\pi,x_{11},x_{12},x_{21})$ and have \[\mbox{gr}_IR_{\bar{\rho}}^{\square,\psi}\cong k[[y_{11},\hat y_{12},y_{21}]][\bar{\pi},\bar{x}_{11},\bar{x}_{12},\bar{x}_{21}].\] Since $\OO=W(k)$ we have $v\in I\setminus I^2$ and hence the elements $I_1,I_2,I_3,I_4$ are homogeneous of degree 2, so that \[\mbox{gr}_IR\cong k[[y_{11},\hat y_{12},y_{21}]][\bar{\pi},\bar{x}_{11},\bar{x}_{12},\bar{x}_{21}]/(I_1,I_2,I_3,I_4),\] see \cite[Ex. 5.3]{MR1322960}. Because $R$ is noetherian it follows from \cite[Cor. 5.5.]{MR1322960} that it is enough to show that $\mbox{gr}_IR$ is an integral domain.

We define \[A:=k[[y_{11},\hat y_{12},y_{21}]][\bar{x}_{11},\bar{x}_{12},\bar{x}_{21},\bar{\pi}]/(\bar{I_1})\] and look at the map
\[\phi\colon A\rightarrow A[\bar{x}_{12}^{-1}]/(\bar{I_2}).\]
The latter ring is isomorphic to $(k[[y_{11},\hat y_{12},y_{21}]][\bar{x}_{11},\bar{x}_{12},\bar{x}_{11}^{-1},\bar{\pi}]/(I_2))$ and since $I_2$ is irreducible it is an integral domain.
So we would be done by showing that $\ker(\phi)=(\bar{I_2},\bar{I_3},\bar{I_4}).$ The inclusion $(I_2,I_3,I_4)\subset\ker(\phi)$ is clear from \eqref{eq:6} and \eqref{eq:8}. For the other one we consider the fact that
\[\ker(\phi)=\{a\in A:\ \exists n\in \NN\cup\{0\},b,c,d\in A: \bar{x}_{12}^na=b\bar{I_2}+c\bar{I_3}+d\bar{I_4}\}.\] To show that $\ker(\phi)\subset (I_2,I_3,I_4),$ we let $a\in A$ and $n$ be minimal with the property that there exist $b,c,d\in A$ such that 
\begin{align}
\bar{x}_{12}^na=b\bar{I_2}+c\bar{I_3}+d\bar{I_4}.\label{eq:9}
\end{align}
If $n=0$ there is nothing to show. Now we assume that $n>0$ and consider the prime ideal $\pp:=(\bar{x}_{12},\bar{v}-\bar{x}_{11})\subset A$ and see that \[A/\pp\cong k[[y_{11},y_{12},y_{21}]][\bar{x}_{11},\bar{x}_{12}]\] is a UFD. We also observe that
\begin{align}
I_2&\equiv y_{12}(\bar{v}+\bar{x}_{11})^2\ \mbox{mod}\ \pp\label{eq:10}\\
I_3&\equiv y_{12}\bar{x}_{21}^2\ \mbox{mod}\ \pp\label{eq:11}\\
I_2&\equiv y_{12}(\bar{v}+\bar{x}_{11})\bar{x}_{21}\ \mbox{mod}\ \pp.\label{eq:12}
\end{align}
Modulo $\pp$ \eqref{eq:9} becomes 
\begin{align}
0\equiv y_{12}b(\bar{v}+\bar{x}_{11})^2+y_{12}c\bar{x}_{21}^2+y_{12}d(\bar{v}+\bar{x}_{11})\bar{x}_{21}.\label{eq:13}
\end{align}
Since $A/\pp$ is a UFD there are $b_1,c_1\in A$ such that 
\begin{align}
y_{12}b&\equiv b_1\bar{x}_{21}\ \mbox{mod}\ \pp\label{eq:14}\\
y_{12}c&\equiv c_1(\bar{v}+\bar{x}_{11})\ \mbox{mod}\ \pp \label{eq:15}
\end{align}
and we see that 
\begin{align}
d\equiv-\frac{b_1\bar{x}_{21}+c_1(\bar{v}+\bar{x}_{11})}{2}\ \mbox{mod}\ \pp.\label{eq:16}
\end{align}
Hence we can find
$b_2, b_3, c_2, c_3, d_1, d_2\in A$ such that
\begin{align*}
b&=b_1\bar{x}_{21}+b_2\bar{x}_{12}+b_3(\bar{v}-\bar{x}_{11})\\
c&=c_1(\bar{v}+\bar{x}_{11})+c_2\bar{x}_{12}+c_3(\bar{v}-\bar{x}_{11})\\
d&=-\frac{b_1\bar{x}_{21}+c_1(\bar{v}+\bar{x}_{11})}{2}+d_1\bar{x}_{12}+d_2(\bar{v}-\bar{x}_{11}).
\end{align*}
Substituting this in \eqref{eq:9} we get
\begin{align}
\bar{x}_{12}^na=&b\bar{I_2}+c\bar{I_3}+d\bar{I_4}\\
\begin{split}
=&\bar{x}_{12}(b_2I_2+b_3I_4+c_2I_3+d_1I_4+d_2I_3)\\
&+\frac{1}{2}(b_1(\bar{v}+\bar{x}_{11})+c_1\bar{x}_{21})I_4+(\bar{v}-\bar{x}_{11})c_3I_3.\label{eq:17}
\end{split}
\end{align}
Modulo $\pp$ we have $b_1(\bar{v}+\bar{x}_{11})+c_1\bar{x}_{21}\equiv 0$ and hence there are $b_4,b_5,b_6,c_4,c_5,c_6$ with
\begin{align}
b_1&=\bar{x}_{21}b_4+\bar{x}_{12}b_5+(\bar{v}-\bar{x}_{11})b_6\\
c_1&=(\bar{v}+\bar{x}_{11})c_4+\bar{x}_{12}c_5+(\bar{v}-\bar{x}_{11})c_6.
\end{align}
Hence we can rewrite \eqref{eq:17} to
\begin{align}
\bar{x}_{12}^na=\bar{x}_{12}z+\frac{1}{2}(b_4+c_4)(\bar{v}+\bar{x}_{11})^2I_3+(\bar{v}-\bar{x}_{11})c_3I_3\label{eq:18}
\end{align}
for a certain $z\in (I_2,I_3,I_4).$
So with \eqref{eq:18} we see that $b_4+c_4\equiv 0$ modulo $\pp$ and $c_3\equiv0$ modulo the prime ideal $\pp':=(\bar{x}_{12},\bar{v}+\bar{x}_{11}).$ Therefore we can find some $c_7,c_8,e_1,e_2\in A$ with
\begin{align*}
c_3&=c_7\bar{x}_{12}+c_8(\bar{v}+\bar{x}_{11})\\
b_4+c_4&=e_1\bar{x}_{12}+e_2(\bar{v}-\bar{x}_{11}).
\end{align*}
But since we have $(v+x_{11})(v-x_{11})=x_{12}x_{21}$ in $A$ we can finally transform \eqref{eq:18} to
\[\bar{x}_{12}^na=\bar{x}_{12}z'\]
for some $z'\in (I_2,I_3,I_4)$ which shows that $\bar{x}_{12}^{n-1}a\in(I_2,I_3,I_4),$ since $A$ is an integral domain. But this a contradiction to the minimality of $n.$
\end{proof}

\begin{Proposition}
$R$ is reduced and $\OO$-torsion free for any choice of $\OO.$
\end{Proposition}

\begin{proof}
Since $\OO$ is flat over $W(k)$ and we have seen in Lemma \ref{Lemma 3} that \[S:=W(k)[[x_{11},\hat x_{12},x_{21},y_{11},\hat y_{12},y_{21}]]/(I_1,I_2,I_3,I_4)\] is an integral domain, we get an injection \[\OO\otimes_{W(k)}S\rightarrow \OO\otimes_{W(k)}\Quot(S).\] As $S$ is $W(k)$-torsion-free by Lemma \ref{Lemma 3}, we obtain an isomorphism \[\OO\otimes_{W(k)}\Quot(S)\xrightarrow{\cong} \OO[\dfrac{1}{p}]\otimes_{W(k)[\frac{1}{p}]}\Quot(S).\] Since $\OO[\frac{1}{p}]$ is a separable field extension of $W(k)[\frac{1}{p}],$ we deduce that $\OO[\frac{1}{p}]\otimes_{W(k)[\frac{1}{p}]}\Quot(S)$ is reduced and $\OO$-torsion free.
\end{proof}

\section{The Multiplicity}\label{sec:Mult}
We want to compute the Hilbert-Samuel-Multiplicity of the ring $R/\pi$ for the given representation 

$$\bar{\rho}\colon G_{\Qp} \rightarrow \GL_2(k), \quad g \mapsto \begin{pmatrix} 1 & \phi(g) \\ 0 & 1 \end{pmatrix}.$$
We denote the maximal ideal of $R/\pi$ by $\mm.$

\begin{Theorem}\label{Theorem 2}
\begin{align*}
e(R/\pi)=\left\{
\begin{array}{l}
1,\ \mbox{if}\ \bar{\rho}\ \mbox{is ramified}\\ 
2,\ \mbox{if}\ \bar{\rho}\ \mbox{is unramified, indecomposable}\\
4,\ \mbox{if}\ \bar{\rho}\ \mbox{is split}
\end{array}
\right.
\end{align*}
\end{Theorem}
\begin{proof}
If we set $J:=y_{12}x_{21}+2x_{11}y_{11}+x_{12}y_{21}$ we obtain modulo $\pi$ the relations
\begin{align}
I_2\equiv&x_{12}J \label{eq:J2}\\
I_3\equiv&x_{21}J \label{eq:J3}\\
I_4\equiv&x_{11}J.\label{eq:J4}
\end{align}
We split the proof into $3$ cases as in Remark \ref{Remark 1}.
If $\bar{\rho}$ is ramified, i.e. $x_{12}$ is invertible, we see as in the proof of Lemma \ref{Lemma 4} that 
\begin{align*}
R/\pi&\cong k[[x_{11},\hat x_{12},x_{21},y_{11},\hat y_{12},y_{21}]]/(x_{11}^2-x_{12}x_{21},J)\\
&\cong k[[x_{11},\hat x_{12},y_{11},\hat y_{12}]].
\end{align*}
Hence it is a regular local ring and therefore $e(R/\pi)=1.$

Let us assume in the following that $\bar{\rho}$ is unramified, i.e. $x_{12}=\hat x_{12}\in \mm_R$, and we can consider the exact sequence
\begin{align}
0\rightarrow (R/\pi)/\mbox{Ann}_{R/\pi}(J)\rightarrow R/\pi\rightarrow R/(\pi,J)\rightarrow 0.\label{eq:s1}
\end{align}
Since $x_{11},x_{12},x_{21}\in \mbox{Ann}_{R/\pi}(J),$ see \eqref{eq:J2}-\eqref{eq:J4}, we have $\dim ((R/\pi)/\mbox{Ann}_{R/\pi}(J))\le 3.$ But $\dim R/\pi=4$ so that \eqref{eq:s1} gives us $e(R/\pi)=e(R/(\pi,J)),$ see \cite[Thm. 14.6]{MR1011461}.
We obtain that
\begin{align*}
R/(\pi,J)\cong& k[[x_{11},x_{12},x_{21},y_{11},\hat y_{12},y_{21}]]/(x_{11}^2-x_{12}x_{21},J)\\
\cong& (k[[x_{11},x_{12},x_{21}]]/(x_{11}^2-x_{12}x_{21}))[[y_{11},\hat y_{12},y_{21}]]/(J)
\end{align*}
 is a complete intersection of dimension 4. So if $\mathfrak{q}\subset R/(\pi,J)$ is an ideal generated by 4 elements, such that $R/(\pi,J,\mathfrak{q})$ has finite length as a $R/(\pi,J)$-module, then these elements form a regular sequence in $R/(\pi,J)$ and $e_{\mathfrak{q}}(R/(\pi,J))=l(R/(\pi,J,\mathfrak{q})),$ see \cite[Thm. 17.11]{MR1011461}. Besides, if there exists an integer $n$ such that $\mathfrak{q}\mm^n=\mm^{n+1},$ then $e(R/(\pi,J))=e_{\mathfrak{q}}(R/(\pi,J)),$ see \cite[Thm. 14.13]{MR1011461}. So to finish the proof it would suffice to find such an ideal $\mathfrak{q}.$

If $\bar{\rho}$ is indecomposable, i.e. $\phi(\delta)$ is non-zero and therefore $y_{12}$ is a unit in $R,$ we can write the equation $J=0$ as \[x_{21}=-y_{12}^{-1}(2x_{11}y_{11}+y_{21}x_{12})\] and $I_1=0$ as \[x_{11}^2=x_{12}y_{12}^{-1}(2x_{11}y_{11}+y_{21}x_{12})\] so that \[R/(\pi,J)\cong k[[x_{11},x_{12},y_{11},\hat y_{12},y_{21}]]/(x_{11}^2+x_{12}y_{12}^{-1}(2x_{11}y_{11}+y_{21}x_{12})).\] Hence it is clear that $x_{12},x_{21},y_{11},\hat y_{12}$ is a system of parameters for $R/(\pi,J)$ that generates an ideal $\mathfrak{q}$ with  $\mathfrak{q}\mm=\mm^2.$ So we obtain
\[e_{\mathfrak{q}}(R/(\pi,J))=l(R/(\pi,J,\mathfrak{q}))=l(k[[x_{11}]]/(x_{11}^2))=2\]
and hence $e(R/\pi)=2.$

If $\bar{\rho}$ is split, which is equivalent to $x_{12},y_{12}\notin R^{\times},$ we take $\mathfrak{q}:=(x_{12}-x_{21},x_{12}-y_{12},x_{12}-y_{21},y_{11})$
and claim that $\mathfrak{q}\mm^2=\mm^3.$ If we write $\mm=(x_{12}-x_{21},x_{12}-y_{12},x_{12}-y_{21},y_{11},x_{11},x_{12})$ we just have to check that $x_{11}^3,x_{11}^2x_{12},x_{11}x_{12}^2,x_{12}^3\in \mathfrak{q}\mm^2.$
Therefore it is enough to see that
\begin{align*}
x_{11}^2=& -x_{11}y_{11}+\frac{1}{2}(x_{12}-y_{12})x_{21}+\frac{1}{2}(x_{21}-y_{21})x_{12}\ \in \mm \mathfrak{q}\\
x_{12}^2=& x_{11}^2+x_{12}(x_{12}-x_{21})\ \in \mm \mathfrak{q}.
\end{align*}
Hence \[e(R/\pi)=l(R/(\pi,J,\mathfrak{q}))=l(k[[x_{11},x_{12}]]/(x_{11}^2,x_{12}^2))=4.\]
\end{proof}

\begin{Corollary}
If $\bar{\rho}$ is unramified, then the ring $R$ is not Cohen-Macaulay.
\end{Corollary}
\begin{proof}
Since $R$ is $\OO$-torsion free, $\pi$ is $R$-regular and hence $R$ is CM if and only if $R/\pi$ is CM. In \eqref{eq:s1} we have constructed a non-zero submodule of $R/\pi$ of dimension strictly less than the dimension of $R/\pi$. It follows from \cite[Thm. 2.1.2(a)]{MR1251956} that $R/\pi$ cannot be CM.
\end{proof}

\begin{Proposition}
Spec$(R/\pi)$ is geometrically irreducible and generically reduced.
\end{Proposition}

To prove the Proposition we need the following Lemma. As in the proof of Theorem \ref{Theorem 2} we define $J:=y_{12}x_{21}+2x_{11}y_{11}+x_{12}y_{21}.$
\begin{Lemma}\label{Lemma 5}
$R/(\pi,J)$ is an integral domain.
\end{Lemma}

\begin{proof}
We again distinguish between 3 cases as in Remark \ref{Remark 1}. If $\bar{\rho}$ is ramified, i.e. $x_{12}$ is invertible, we have already seen in the proof of Theorem \ref{Theorem 2} that
\begin{align*}
R/(\pi,J)&\cong k[[x_{11},\hat x_{12},x_{21},y_{11},\hat y_{12},y_{21}]]/(x_{11}^2-x_{12}x_{21},J)\\
&\cong k[[x_{11},\hat x_{12},y_{11},\hat y_{12}]].
\end{align*}
If $\bar{\rho}$ is unramified and indecomposable, i.e. $x_{12}=\hat x_{12}\in \mm_R, y_{12}\in R^{\times}$ we saw that
\[R/(\pi,J)\cong k[[x_{11},x_{12},y_{11},\hat y_{12},y_{21}]]/(x_{11}^2+x_{12}y_{12}^{-1}(2x_{11}y_{11}+y_{21}x_{12}))\] which is easily checked to be an integral domain.
If $\bar{\rho}$ is unramified and split, i.e. $x_{12},y_{12}\in \mm_R,$ let $\nn$ denote the maximal ideal of $R/(\pi,J).$ It is enough to show that the graded ring $\mbox{gr}_{\nn}R/(\pi,J)$ is a domain. Since $J$ is homogeneous we have
\[\mbox{gr}_{\nn}R/(\pi,J)\cong k[x_{11},x_{12},x_{21},y_{11},y_{12},y_{21}]/(x_{11}^2-x_{12}x_{21}, J).\]
We set $A:=k[x_{11},x_{12},x_{21},y_{11},y_{12},y_{21}]/(x_{11}^2-x_{12}x_{21})$ and have to prove that $(J)\subset A$ is a prime ideal. We look at the localization map $A\xrightarrow{\iota} A[y_{21}^{-1}],$ which is an inclusion because $y_{21}$ is regular in $A.$ This gives us a map $A \xrightarrow{\bar{\iota}} A[y_{21}^{-1}]/(J).$ Since \[A[y_{21}^{-1}]/(J)\cong k[x_{11},x_{21},y_{11},y_{12},y_{21},y_{21}^{-1}]/(x_{11}^2+x_{21}y_{21}^{-1}(2x_{11}y_{11}+x_{21}y_{12}))\] is a domain, we would be done by showing that $\ker(\bar{\iota})=(J).$ We have \[\ker(\bar{\iota})=\{a\in A: y_{21}^ia=bJ\ \mbox{for some}\ i\in \ZZ_{\geq0}, b\in A : y_{21}\nmid b\}.\] But since $(y_{21})\subset A$ is a prime ideal and $y_{21}$ does not divide $J,$ we see that $i=0$ in all these equations and hence $\ker(\bar{\iota})=(J).$
\end{proof}

\begin{proof}[Proof of the Proposition]
Let $\pp$ be a minimal prime ideal of $S:=R/\pi.$
It follows from (\ref{eq:J2})-(\ref{eq:J4}) that $J^2=0$ and thus $J\in \mbox{rad}(S)=\bigcap_{\pp\ \mbox{minimal}}\pp.$ So Lemma \ref{Lemma 5} gives us that $JS$ is the only minimal prime ideal of $S,$ hence Spec$(S)$ is irreducible. If we replace the field $k$ by an extension $k',$
we obtain the irreducibility of Spec$(S\otimes_kk')$ analogously, thus Spec$(S)$ is geometrically irreducible.

Spec$(S)$ is called generically reduced if $S_{\pp}$ is reduced for any minimal prime ideal $\pp.$ We have already seen that there is only one minimal prime ideal $\pp=JS.$ By localizing (\ref{eq:s1}) we obtain $S_{\pp}\cong R/(\pi,J).$ Lemma \ref{Lemma 5} implies that $S_{\pp}$ is reduced.

\end{proof}

\subsection*{Acknowledgement}
This paper was written as a part of my PhD-thesis. I want to thank my advisor Prof. Dr. Vytautas Pa\v{s}k\={u}nas for his great support.

\bibliographystyle{mrl}
\bibliography{References}

\end{document}